\documentclass[12pt]{amsart}
\usepackage{verbatim}
\usepackage{amssymb, amsmath}
\usepackage{graphicx}
\usepackage{caption}
\usepackage{subcaption}
 \usepackage{enumitem}
\usepackage{color}
\usepackage{url}

 \usepackage[colorlinks,linkcolor=blue,anchorcolor=blue,citecolor=blue,backref=page]{hyperref}

\usepackage{hyperref}

\begin{document}
\newtheorem{thm}{Theorem}[section]
\newtheorem*{thm*}{Theorem}
\newtheorem{lem}[thm]{Lemma}
\newtheorem{prop}[thm]{Proposition}
\newtheorem{cor}[thm]{Corollary}
\newtheorem{conj}{Conjecture}
\newtheorem{proj}[thm]{Project}
\newtheorem{question}{Question}
\newtheorem{rem}{Remark}[section]

\theoremstyle{definition}
\newtheorem*{defn}{Definition}
\newtheorem*{remark}{Remark}
\newtheorem{exercise}{Exercise}
\newtheorem*{exercise*}{Exercise}

\numberwithin{equation}{section}

\newcommand{\rad}{\operatorname{rad}}

\newcommand{\Z}{{\mathbb Z}} 
\newcommand{\Q}{{\mathbb Q}}
\newcommand{\R}{{\mathbb R}}
\newcommand{\C}{{\mathbb C}}
\newcommand{\N}{{\mathbb N}}
\newcommand{\FF}{{\mathbb F}}
\newcommand{\fq}{\mathbb{F}_q}
\newcommand{\rmk}[1]{\footnote{{\bf Comment:} #1}}

\renewcommand{\mod}{\;\operatorname{mod}}
\newcommand{\ord}{\operatorname{ord}}
\newcommand{\TT}{\mathbb{T}}
\renewcommand{\i}{{\mathrm{i}}}
\renewcommand{\d}{{\mathrm{d}}}
\renewcommand{\^}{\widehat}
\newcommand{\HH}{\mathbb H}
\newcommand{\Vol}{\operatorname{vol}}
\newcommand{\area}{\operatorname{area}}
\newcommand{\tr}{\operatorname{tr}}
\newcommand{\norm}{\mathcal N} 
\newcommand{\intinf}{\int_{-\infty}^\infty}
\newcommand{\ave}[1]{\left\langle#1\right\rangle} 
\newcommand{\Var}{\operatorname{Var}}
\newcommand{\Prob}{\operatorname{Prob}}
\newcommand{\sym}{\operatorname{Sym}}
\newcommand{\disc}{\operatorname{disc}}
\newcommand{\CA}{{\mathcal C}_A}
\newcommand{\cond}{\operatorname{cond}} 
\newcommand{\lcm}{\operatorname{lcm}}
\newcommand{\Kl}{\operatorname{Kl}} 
\newcommand{\leg}[2]{\left( \frac{#1}{#2} \right)}  
\newcommand{\Li}{\operatorname{Li}}

\newcommand{\sumstar}{\sideset \and^{*} \to \sum}

\newcommand{\LL}{\mathcal L} 
\newcommand{\sumf}{\sum^\flat}
\newcommand{\Hgev}{\mathcal H_{2g+2,q}}
\newcommand{\USp}{\operatorname{USp}}
\newcommand{\conv}{*}
\newcommand{\dist} {\operatorname{dist}}
\newcommand{\CF}{c_0} 
\newcommand{\kerp}{\mathcal K}

\newcommand{\Cov}{\operatorname{cov}}
\newcommand{\Sym}{\operatorname{Sym}}

\newcommand{\Ht}{\operatorname{Ht}}

\newcommand{\E}{\operatorname{\mathbb E}} 
\newcommand{\sign}{\operatorname{sign}} 
\newcommand{\meas}{\operatorname{meas}} 
\newcommand{\length}{\operatorname{length}} 

\newcommand{\divid}{d} 

\newcommand{\GL}{\operatorname{GL}}
\newcommand{\SL}{\operatorname{SL}}
\newcommand{\re}{\operatorname{Re}}
\newcommand{\im}{\operatorname{Im}}
\newcommand{\res}{\operatorname{Res}}
 \newcommand{\eigen}{\Lambda} 
\newcommand{\tens}{\mathbf t} 
\newcommand{\diam}{\operatorname{diam}}
\newcommand{\fixme}[1]{\footnote{Fixme: #1}}
 \newcommand{\EWp}{\mathbb E^{\rm WP}_g} 
\newcommand{\orb}{\operatorname{Orb}}
\newcommand{\supp}{\operatorname{Supp}}
\newcommand{\mmfactor }{\textcolor{red}{c_{\rm Mir}}}
\newcommand{\Mg}{\mathcal M_g} 
\newcommand{\MCG}{\operatorname{Mod}} 
\newcommand{\Diff}{\operatorname{Diff}} 
\newcommand{\If}{I_f(L,\tau)}

\newcommand{\GOE}{\operatorname{GOE}}
\newcommand{\GUE}{\operatorname{GUE}}
\newcommand{\GSE}{\operatorname{GSE}}
\newcommand{\rest}{\operatorname{rest}} 
\newcommand{\diag}{\operatorname{DIAG}} %
\newcommand{\off}{\operatorname{OFF}} %

\title[Hammarhjelm's condition]{The classification of real quadratic fields which satisfy Hammarhjelm's condition}
\author{Ze\'ev Rudnick}
\address{School of Mathematical Sciences, Tel Aviv University, Tel Aviv 69978, Israel} 
\email{rudnick@tauex.tau.ac.il}
\thanks{This research was supported by the Israel Science Foundation (grant No. 2860/24).}
\date{\today}
 \begin{abstract}
A real quadratic field satisfies Hammarhjelm's condition if its ring of integers has unique factorization, and the Minkowski lattice of its ring of integers contains no point in a certain rectangle determined by the fundamental unit. Such fields have recently appeared in the study of visible points in algebraic cut-and-project sets. We prove that there are exactly seven real quadratic fields satisfying Hammarhjelm's condition, namely those with discriminant
\[
8,\;5,\;13,\;29,\;53,\;173,\;293.
\]
The proof is based on showing that for such fields, the fundamental unit is small relative to the discriminant, together with genus theory and Bir\'o's classification of class number one fields in Yokoi's family.
\end{abstract}
\maketitle

 \tableofcontents
\section{Hammarhjelm's condition}

Let $K$ be a real quadratic field, $K=\Q(\sqrt{D})$ with $D>1$ squarefree, with Galois involution $\sigma:K\to K$ and ring of integers $O_K$. 
Let $\lambda>1$ be the fundamental unit of $K$, the unique element $\lambda>1$ so that the unit group of $O_K$ is $O_K^\times = \{\pm \lambda^n, n\in \Z  \}$. 

The Minkowski embedding 
\[
f: K\hookrightarrow \R^2, \quad \alpha\mapsto \left(\alpha,\sigma(\alpha) \right)
\]
  embeds $O_K$ as a lattice $f(O_K)$ in $\R^2$.

Following \cite{GKW}, we say that $K$ satisfies Hammarhjelm's condition  if
\begin{enumerate}
\item $O_K$ has unique factorization
\item  Condition H holds: 
\begin{equation}\label{eq:the H condition}
f(O_K)\cap\bigl((1,\lambda)\times(-1,1)\bigr)=\emptyset .
\end{equation}
\end{enumerate}

The known examples until now were $D=2, 5, 13, 29, 53$   \cite{Hammarhjelm, GKW}. 



Real quadratic fields satisfying Hammarhjelm's condition have recently played a central role in the study of visible points in certain cut-and-project sets \cite{GKW, KO}, motivating the problem of determining all such fields. 
Our result is:

\begin{thm}\label{main thm}
There are exactly seven real quadratic fields $K=\Q(\sqrt{D})$ which satisfy Hammarhjelm's condition,  namely those with 
\[
D=2,5, 13,  29, 53, 173, 293.
\]
\end{thm}

The cases $D=2,5, 13, 29, 53$ were previously known  \cite{Hammarhjelm, GKW}.

About the proof: The main point is to realize that Hammarhjelm's condition implies that the regulator of the field is small relative to the discriminant, so the class number formula forces the class number to be large. This can be converted to a proof of finiteness of such fields using Siegel's theorem, which is  ineffective. To pin down the full list of extensions, we use a number of steps, with one of the cases reduced to a special family of fields where those with class number one are all known: Yokoi's family, with squarefree discriminant of the form $D=m^2+4$, classified by Bir\'o  in 2003  \cite{Biro}. 

 \section{Background on real quadratic fields}
 
 We may write $K=\Q(\sqrt{D})$ with $D>1$ squarefree, and then the discriminant of $K$ is 
\[
\Delta_K = \begin{cases}  4D,& D=2,3 \bmod 4\\ D,& D=1 \bmod 4.
\end{cases}
\]

The ring of integers of $K$ is $O_K=\Z[\omega]$, where
\[
\omega=\omega_D = \begin{cases} \sqrt{D},& D=2,3 \bmod 4\\ \frac{1+\sqrt{D}}2,& D=1 \bmod 4 . \end{cases}
\]

\subsection{The fundamental unit}

 The following Lemma is well known, a proof is included for the reader's convenience.
 \begin{lem}\label{lem:norm} 
Let $D>1$ be squarefree with $D\equiv 1 \pmod 4$. 
Write the fundamental unit $\lambda>1$ in the form
\[
\lambda=m+n\omega,
\qquad m,n\in\Z.
\]
Then
\[
m\ge 0,
\qquad
n\ge 1.
\]
\end{lem}

\begin{proof}
First note that $n\neq 0$. Indeed, if $n=0$, then $\lambda=m\in\Z$, and since
$\lambda$ is a unit, we would have $\lambda=\pm1$, contradicting $\lambda>1$.

Suppose now that $n<0$. Writing
\[
\lambda=\left(m+\frac n2\right)+\frac n2\sqrt D,
\]
we see that
\[
\sigma(\lambda)
=
\left(m+\frac n2\right)-\frac n2\sqrt D
=
\left(m-\frac{|n|}{2}\right)+\frac{|n|}{2}\sqrt D.
\]
Since $\sqrt D>1$, it follows that
\[
\sigma(\lambda)>\lambda>1.
\]
On the other hand, because $\lambda$ is a unit,
\[
N(\lambda)=\lambda\,\sigma(\lambda)=\pm1,
\]
and therefore
\[
|\sigma(\lambda)|=\frac1{\lambda}<1,
\]
a contradiction. Hence $n>0$, and since $n$ is integral, we obtain
\[
n\ge1.
\]

Next, because $\lambda>1$ is a unit, we again have
\[
|\sigma(\lambda)|<1.
\]
Writing
\[
\sigma(\lambda)
=
m+\frac n2-\frac n2\sqrt D,
\]
the inequality $|\sigma(\lambda)|<1$ implies
\[
-1
<
m+\frac n2-\frac n2\sqrt D
<
1.
\]
From the left-hand inequality we obtain
\[
m>\frac n2(\sqrt D-1)-1.
\]
Since $D\ge5$ and $n\ge1$, the right-hand side is strictly greater than $-1$. Thus
$m>-1$. As $m$ is an integer, it follows that
\[
m\ge0.
\]

Therefore  $\lambda=m+n\omega$ with $m\ge 0$, $n\ge1$.
\end{proof}
 
\subsection{Narrow vs wide equivalence}
For a real quadratic field $K$, there is a distinction between narrow and wide equivalences of ideals, that is two fractional ideals are narrowly equivalent if there is an element  $\alpha\in K$  of positive norm $N(\alpha)>0$, so that $I=(\alpha) J$, and equivalent in the ordinary sense without the restriction that $\alpha$ has positive norm.  If there a unit of negative norm, then the two notions coincide, otherwise they are distinct.

Denote  by ${\rm Cl}(K)$ the ideal classes, and by ${\rm Cl}^+(K)$ the narrow ideal classes. 
The narrow class number is $h_K^+ =\# {\rm Cl}^+(K) $, and the wide class number is  $h_K=\#{\rm Cl}(K)$. 
If there is a  unit of negative norm, then $h_K^+=h_K$, and otherwise $h_K^+=2h_K$.

\subsection{Genus theory}
 We recall the restrictions imposed by genus theory on the prime factorization of the discriminant of a real quadratic field with class number one.

According to genus theory, the narrow class number $h_K^+$ is divisible by $2^{r(K)-1}$,  
where the integer $r(K)$ is the number of distinct prime factors of the discriminant $\Delta_K$; recall $\Delta_K=4D$ if $D=2,3 \bmod 4$, and $\Delta_K=D$ if $D=1 \bmod 4$.

Hence a necessary condition for $K=\Q(\sqrt{D})$ to have narrow class number one is one of the  following two cases:
\begin{enumerate}
\item $D=1 \bmod 4$ then $\Delta_K = D=p$ is prime.
\item $D=2 \bmod 4$ then the only option is $D=2$, $\Delta_K=8$. 
\end{enumerate}

If there is a unit of negative norm, then $h_K=h_K^+$ and the cases above determine when the class number is one. However, if  there are no units with negative norm, so that $h_K^+=2h_K$,  then we will need to treat the case when $r(K)=2$. Of concern to us in this case will be when $D=1 \bmod 4$, so that the discriminant $\Delta=D=pq$ is a product of two distinct odd primes.

\subsection{The $4$-rank of the narrow class group}


The $4$-rank of a finite abelian group $A$ is the number   of cyclic summands of $A$ whose order is divisible by $4$. For a discriminant $\Delta$, denote by $e_4(\Delta)$ the $4$-rank of the narrow ideal class group of $K=\Q(\sqrt{\Delta})$.  So if $e_4(\Delta)>0$ then $h_K^+$ is divisible by $4$, and therefore  $h_K>1$.  R\'edei  and Reichardt \cite{Redei-Reichardt, Redei} showed
that  
\[
e_4(\Delta)= r(K)-1-{\rm rank}\, M(\Delta)
\]
where $M(\Delta)=(\varepsilon_{ij})_{i,j=1}^{r(K)}$ is the $r(K)\times r(K)$   matrix defined over $\FF_2$, whose entries are given by 
\[
(-1)^{\varepsilon_{ij}} = \leg{d_i}{d_j},\quad i\neq j,
\] 
where $\Delta=\prod_{i=1}^{r(K)} d_i$ is the decomposition of the discriminant into prime power discriminants, we set $\leg{x}{d_j} = \leg{x}{2} = (-1)^{(x-1)/4}$ if $d_j$ is even, and 
\[
\varepsilon_{jj} =\sum_{\substack{i\\i\neq j}} \varepsilon_{ij}.
\]

\begin{lem}\label{lem:4-rank}
Let $\Delta=pq$ with $p$, $q$ distinct primes, $p\equiv q\equiv 1 \bmod 4$ and $q=p+4$.
Then the $4$-rank of the narrow class group is one, and hence the class number is at least $2$.
\end{lem}
\begin{proof}
In this case,  $r=2$, $d_1=p$, $d_2=q$, and by quadratic reciprocity 
\[
\leg{d_1}{d_2} = \leg{p}{q}=\leg{q}{p}=\leg{d_2}{d_1}
\]
and hence all entries in $M$ are equal: $M=\left(\begin{smallmatrix} u&u\\u&u \end{smallmatrix}\right)$. 

Assume further that $q=p+4$. Then $\leg{q}{p} = \leg{p+4}{p}=\leg{4}{p} = +1$ so that $(-1)^{\varepsilon_{12}} = 1 = (-1)^{\varepsilon_{21}}$ hence $\varepsilon_{12 }=0=\varepsilon_{21}$ and the R\'edei matrix is 
\[
M(pq) = \begin{pmatrix}  0&0\\0&0 \end{pmatrix}
\]
which has rank zero over $\FF_2$. Therefore the $4$-rank of the narrow class group is
\[
e_4(pq)= r(K)-1-{\rm rank}\, M(pq)=2-1-0=1.
\]
We deduce that $4\mid h_K^+$, and since $h_K$ equals $h_K^+$ or $h_K^+/2$,   we cannot have $h_K=1$.
\end{proof}

\section{Condition H}
Recall  condition H: 
\[
f(O_K)\cap\bigl((1,\lambda)\times(-1,1)\bigr)=\emptyset .
\]
We explore some consequences of condition H.

We first show that condition H forces the fundamental unit $\lambda$ to be small relative to the discriminant $\Delta_K$ of the field:

\begin{prop}\label{prop H}
Let $K=\mathbf{Q}(\sqrt D)$ be a real quadratic field, where $D>1$ is squarefree, and let
$\lambda>1$ be the fundamental unit of $K$. Assume that $K$ satisfies condition H. 
Then:
\begin{enumerate}
\item If $D\equiv 2,3 \pmod 4$, then
\[
\lambda< 2\sqrt D=\sqrt{\Delta_K}.
\]

\item If $D\equiv 1 \pmod 4$, then
\[
\lambda <\sqrt D+\frac 12=\sqrt{\Delta_K}+\frac 12.
\]
\end{enumerate}
\end{prop}

\begin{proof}
Suppose first that $D\equiv 2,3 \pmod 4$. Then
\[
O_K=\Z[\sqrt D].
\]
Assume, for contradiction, that
\[
\lambda\geq 2\sqrt D.
\]
Let
\[
\alpha=\lfloor\sqrt D\rfloor+\sqrt D \in O_K.
\]
Since $D>1$ is squarefree, $\sqrt D\notin\Z$, and hence
\[
\sigma(\alpha) = \lfloor\sqrt D\rfloor-\sqrt D = -\{\sqrt D\},
\]
where $\{\sqrt D\}\in(0,1)$ denotes the fractional part of $\sqrt D$. Thus
\[
-1<\sigma(\alpha)<0.
\]
Moreover,
\[
1<\alpha = \lfloor\sqrt D\rfloor+\sqrt D < 2\sqrt D \leq \lambda
\]
by our assumption. 
Therefore
\[
f(\alpha)\in (1,\lambda)\times(-1,1) 
\]
which contradicts condition H. Hence we must have
\[
\lambda< 2\sqrt D.
\]

 Suppose now that $D\equiv 1 \pmod 4$. Then
\[
O_K=\Z[\omega],
\qquad
\omega=\frac{1+\sqrt D}{2}.
\]
We claim that condition H implies
\[
\lambda <\sqrt D+\frac12.
\]

Assume, for contradiction, that
\[
\lambda \geq \sqrt D+\frac12.
\]
Let
\[
x=\frac{\sqrt D-1}{2},
\qquad
m=\left\lfloor x+\frac12\right\rfloor= \left\lfloor \frac{\sqrt{D}}{2} \right\rfloor,
\]
and define
\[
\alpha=m+\omega.
\]

Since
\[
\omega'=\sigma(\omega)=\frac{1-\sqrt D}{2}=-x,
\]
we have
\[
\sigma(\alpha)=m-x.
\]
By construction, $m$ is the nearest integer to $x$, and therefore
\[
|\sigma(\alpha)|=|m-x|\le \frac12<1.
\]

On the other hand,
\[
m<x+\frac12,
\]
so
\[
\alpha
=
m+\omega
<
x+\frac12+\omega
=
\frac{\sqrt D-1}{2}
+\frac12
+\frac{1+\sqrt D}{2}
=
\sqrt D+\frac12.
\]
Hence
\[
\alpha<\sqrt D+\frac12\leq \lambda.
\]

Since $D\ge 5$, we have $\omega>1$, and therefore
\[
\alpha=m+\omega>1.
\]
Thus
\[
1<\alpha<\lambda
\qquad\text{and}\qquad
|\sigma(\alpha)|<1.
\]
Consequently,
\[
f(\alpha)\in (1,\lambda)\times(-1,1) 
\]
which contradicts condition H. 
Therefore
\[
\lambda< \sqrt D+\frac12.
\]
This completes the proof.
\end{proof}

\section{Yokoi's family}

Yokoi's family of real quadratic fields is  $\Q(\sqrt{D})$ with $D=x^2+4$ squarefree. 
We show that they all satisfy condition H.
 
\begin{prop}
 Let $D=x^{2}+4>5$ be squarefree, where $x=2m+1$ is odd, and let
$K=\Q(\sqrt D)$. 
Then $K$ satisfies condition H. 
\end{prop}
 
\begin{proof}
The ring of integers here is $O_K = \Z[\omega]$, $\omega=\frac{1+\sqrt D}{2}$.
 Since $D=x^2+4$, we have $x^{2}-D=-4$, and therefore
\[
\lambda_1=\frac{x+\sqrt D}{2}=m+\omega
\]
is a unit of norm $-1$, with $\lambda_1>1$.  Hence it is a positive (odd) power of the fundamental unit $\lambda$, in particular $\lambda_1\geq \lambda$. 

Hence it suffices to show that there is no nonzero element
\[
\alpha=a+b\omega\in O_K=\Z[\omega]
\]
such that
\[
1<\alpha<\lambda_1
\qquad\text{and}\qquad
|\sigma(\alpha)|<1.
\]

Let
\[
\omega'=\sigma(\omega)=\frac{1-\sqrt D}{2}
\]
be the Galois conjugate. Then we require
\[
|\sigma(\alpha)| =|a+b\omega'| <1.
\]
Since $\omega+\omega'=1$, we may write
\[
\alpha = (a+b\omega')+b(\omega-\omega')
=
  \sigma(\alpha) +b\sqrt D.
\]
Hence, since $|\sigma(\alpha)|<1$, 
\[
\alpha>b\sqrt D-1.
\]

On the other hand,
\[
\lambda_1=\frac{x+\sqrt D}{2}<\sqrt D,
\]
because $x<\sqrt D=\sqrt{x^2+4}$. 
If $b\geq 2$, then
\[
\alpha>b\sqrt D-1
\ge 2\sqrt D-1
> \lambda_1,
\]
a contradiction. Therefore $b\le1$.

Next suppose $b\le0$. Since $\omega'<0$, the condition
$|a+b\omega'|<1$ implies $a\le 0$, because otherwise $a\geq 1$ by integrality, and then $a+b\omega'\geq 1$. Consequently,
\[
\alpha=a+b\omega\le0,
\]
contrary to $\alpha>1$. Hence $b=1$.

Thus every potential counterexample must be of the form
\[
\alpha=a+\omega.
\]

The condition $|\sigma(\alpha)|<1$ becomes
\[
|a+\omega'|<1.
\]
which is equivalent to
\[
-\omega'-1<a<-\omega'+1.
\]
Now
\[
-\omega'=\frac{\sqrt D-1}{2}.
\]
Since $D=x^2+4$ and $x=2m+1$, we have
\[
m=\frac{x-1}{2}
<
\frac{\sqrt D-1}{2}
<
m+\frac12.
\]
Therefore
\[
m-1<-\omega'-1<m-\frac12,
\]
and
\[
m+1<-\omega'+1<m+\frac32.
\]
Hence the interval
\[
(-\omega'-1,-\omega'+1)
\]
contains exactly the two integers $m$ and $m+1$. It follows that
\[
a=m \quad\text{or}\quad a=m+1.
\]

On the other hand,
\[
\alpha=a+\omega<\lambda_1=m+\omega
\]
implies
\[
a<m.
\]
This contradicts \(a\in\{m,m+1\}\).

Therefore there is no element $\alpha\in O_K$ satisfying
\[
1<\alpha<\lambda_1
\qquad\text{and}\qquad
|\sigma(\alpha)|<1.
\]

Hence $K$ satisfies condition H.
\end{proof}

 Bir\'o \cite{Biro} showed that the only members of Yokoi's family when the class number is one are those where $D=5, 13, 29, 53, 173, 293$. 
Hence we find:
\begin{cor}\label{cor:Yokoi} 
Among Yokoi's family, the fields satisfying both condition H and class number one are precisely
\[
D=5, 13, 29, 53, 173, 293.
\]
\end{cor}

\section{Extensions of narrow Richaud-Degert type}
If $D=n^2-4>0$ is square-free (so $n=2m+1\geq 3$ is odd), then the corresponding real quadratic field $K=\Q(\sqrt{D})$ is called a narrow Richaud-Degert extension. 
Degert \cite[Satz 1]{Degert} showed\footnote{ \cite[Satz 1]{Degert} showed that the smallest unit $u>1$ of positive norm is $m+\omega$, but it is easy to see that there is no solution to $\lambda=\alpha^2$ for $\alpha\in O_K$, so this is indeed the fundamental unit. Alternatively, note   the continued fraction $\sqrt{\omega}  = [m;\overline{1, 2m-1}]$, which has even period \cite[Theorem 2.1.3]{Mollin}. } that  if  $D=n^2-4>5$ then fundamental unit is $\lambda =m+\omega=m+\frac 12 +\frac 12 \sqrt{D}$, which has positive norm.


We show that these extensions do not satisfy condition H if $D>5$: 
\begin{prop}\label{prop:Degert}
Let $D=(2m+1)^2-4$ be squarefree. Then $K=\Q(\sqrt{D})$ does not satisfy condition H.
\end{prop}
\begin{proof}
Recall that the fundamental unit of $K$ is 
\[
\lambda=m+\omega.
\]
Let 
\[
\alpha = m-1+\omega=\lambda-1.
\]
Then clearly
\[
1<\alpha<\lambda.
\]
We will show that the Galois conjugate
\[
\sigma(\alpha)= m-\frac 12 -\frac{\sqrt{D}}{2}\in (-1,-\frac 12)
\]
and therefore we found an element in $f(O_K)\cap ((1,\lambda)\times (-1,1))$.

We have
\[
(2m)^2<D=(2m+1)^2-4<(2m+1)^2
\]
if $m\geq 1$, therefore
\[
-m-\frac 12< -\frac{\sqrt{D}}{2}<-m
\]
and so 
\[
\sigma(\alpha) =  m-\frac 12 -\frac{\sqrt{D}}{2}\in  m-\frac 12 +(-m-\frac 12,-m) = (-1,-\frac 12)
\]
as claimed. 
\end{proof}

\section{Proof of Theorem~\ref{main thm}}
We want to determine a complete list of all real quadratic fields that  satisfy  Hammarhjelm's condition.

We divide into a number of cases. 

\subsection{$D=2 \bmod 4$}
When $D=2$, the class number is one, and indeed $\Q(\sqrt{2})$ satisfies Hammarhjelm's condition. 

It remains to consider $D>2$, $D=2 \bmod 4$. 
We claim that there are no such cases. 

To have class number one with $D=2 \bmod 4$, the narrow class number has to be either $h_K^+=1$, forcing $D=2$, or $h_K^+=2$, and in the latter case we need $D=2p$ with $p$ an odd prime, by genus theory. 
If there is a unit of negative norm then $h_K=h_K^+$ is therefore even, so we cannot have class number one. 
So we only need to consider the case that $D=2p$ and there is no unit of negative norm: $N(\lambda)=+1$. 

Writing $\lambda=t+\sqrt{D}$ gives $t^2-D=1$ ($t$ must be odd) or 
\[
2p=D = (t-1)(t+1).
\]
Since $t^2=D+1=3\bmod 4$, $t$ must be odd, and so 
\[
t^2-1=(t-1)(t+1)=0\bmod 4,
\]
 so cannot be of the form $2p$ with $p$ an odd prime.

\subsection{$D=3 \bmod 4$} 
We claim that there are no such cases when  Hammarhjelm's  condition holds. 

The ring of integers is $O_K=\Z[\sqrt{D}]$, so the fundamental unit is $\lambda = t+u\sqrt{D}$. 
Note that the negative Pell equation $t^2-Du^2=-1$ has no solution if $D=3 \bmod 4$ \footnote{Indeed, if $D=3 \bmod 4$ then it must be divisible by a prime $p=3 \bmod 4$, and then $t^2-Du^2=-1$ implies   $t^2=-1 \bmod p$, 
which rules out $p=3 \bmod 4$.}, hence $K$ has no unit of negative norm. Thus  $N(\lambda)=+1$ and $h_K^+=2h_K$.  

Since $D=3 \bmod 4$, the discriminant $\Delta_K=4D$ is divisible by $2$   and by at least one odd prime discriminant, hence $r(K)\geq 2$.  
In order for the class number $h_K=1$ we need $r(K)=2$, which  will happen if and only if $D=p$ is a prime (recall $D$ is squarefree). In conclusion, for $h_K=1$ we need $D=p$ is prime, $p=3 \bmod 4$. 

Using Proposition~\ref{prop H}, that $\lambda = t+u\sqrt{D}< 2\sqrt{D}$ $(t,u\geq 1$), we deduce that $u=1$ (if $u\geq 2$ then since $t\geq 0$, necessarily $\lambda\geq 2\sqrt{D}$), so that 
\[
\lambda=t+\sqrt{D}, \quad t>0,
\]
 and the condition $N(\lambda)=+1$ gives $t^2-D=1$, that is 
\[
p=D=t^2-1=(t-1)(t+1).
\]
But $D=p$ is an odd prime, which forces  $t$ even, hence the only possibility is $t=2$ and $D=3$.  But then the fundamental unit  $\lambda=2+\sqrt{3}$ does not satisfy condition H since by Proposition~\ref{prop H} that would force $\lambda< 2\sqrt{3}$, while  $\lambda=2+\sqrt{3}>2\sqrt{3}$.

\section{The case $D=1 \bmod 4$} 

\begin{lem}\label{fund unit when D=1 mod 4}
If $D=1 \bmod 4$ is squarefree, and the field $\Q(\sqrt{D})$ satisfies condition H, then the fundamental unit is 
\begin{equation}\label{shape of fund unit}
\lambda = m+\omega, \quad 0\leq m<\frac{\sqrt{D}}{2}.
\end{equation}
\end{lem}
\begin{proof}
If $D=1 \bmod 4$, then $O_K=\Z[\omega_K]$, $\omega_K=\frac{1+\sqrt{D}}{2}$, the discriminant is $\Delta_K=D$, and the fundamental unit is $\lambda=m+n\omega_K$ with   
$m\geq 0$, $n\geq 1$ (Lemma~\ref{lem:norm}).   
Condition $H$ is 
\[
\lambda<\sqrt{D}+\frac 12 
\]
or
\[
m+\frac n2 +\frac{n}{2}\sqrt{D}< \sqrt{D}+\frac 12 
\]
giving
\[
\left( m+\frac{n-1}{2} \right) + \frac{n-2}{2}\sqrt{D}<0.
\]
If $n\geq 2$ then this cannot happen, because the LHS is strictly positive (recall also $m\geq 0$). Hence $n=1$ is the only option, so 
\[
\lambda = m+\omega_K = m+\frac 12 +\frac{\sqrt{D}}{2}.
\]
Condition H  then implies
\[
\lambda=m+\frac 12 +\frac{\sqrt{D}}{2}< \sqrt{D}+\frac{1}{2}, \quad m\geq 0
\]
which is equivalent to $0\leq m<\frac{\sqrt{D}}{2}$. Hence we find that 
\begin{equation*}
\lambda = m+\omega, \quad 0\leq m<\frac{\sqrt{D}}{2}.
\end{equation*}
\end{proof}

To have class number one we need  $r(K)\leq 2$ since $2^{r(K)-1}\mid h_K^+$. Thus when $D=1 \bmod 4$, either

\begin{enumerate}
\item $\Delta_K=D=p$ is a prime, $p\equiv 1\pmod 4$; or

\item $\Delta_K=D=pq$ is a product of two distinct primes. Then $r(K)=2$, so $h_K^+$ is even, and if $h_K=1$ then necessarily $h_K^+=2$, which implies that there is no unit of negative norm.
\end{enumerate}

\subsection{Prime discriminant}
We show that when $D=p$ is a prime $p=1 \bmod 4$ and condition H holds then the fundamental unit must have norm $-1$. Indeed, by Lemma~\ref{fund unit when D=1 mod 4}, the fundamental unit is   $\lambda = m+\omega = \frac{2m+1+\sqrt{p}}{2}$,  with $m\geq 0$. 
If $N(\lambda)=+1$ then 
\[
+1 = N(\lambda)=\frac{(2m+1)^2-p}{4} \; \longleftrightarrow \;  p=(2m+1)^2-4 =(2m-1)(2m+3)
\]
and since $p$ is prime, this forces $m=1$ and $p=5$. For   $m=1$, when $p=5$, the unit $m+\omega=1+\omega$ is not the fundamental unit. 
The fundamental unit is $\omega = (1+\sqrt{5})/2$,  which has norm $-1$.

Hence we deduce that $N(\lambda)=-1$, and since $\lambda=m+\omega_K$, we obtain  the equation
\[
-1= N(\lambda) = \frac{(2m+1)^2-p}{4} \quad \longleftrightarrow \quad  p=(2m+1)^2+4
\]
that is
\[
D=p=(2m+1)^2+4. 
\]
Thus we are  a subset  of Yokoi's family,  and we showed in Corollary~\ref{cor:Yokoi}  that the discriminants in Yokoi's family which satisfy Hammarhjelm's condition are precisely those with
\[
D=5, 13, 29, 53, 173, 293.
\]
These are all cases of prime discriminants   where Hammarhjelm's condition holds.

\subsection{$D=pq$}
Next, assume that $\Delta=D=pq$ with $p,q$ distinct primes, $p=q \bmod 4$. 

  Case 1: $p=q=1 \bmod 4$: 
We showed (Lemma~\ref{lem:4-rank}) that the $4$-rank of the narrow class group is $1$, hence $4\mid h_K^+$ and $h_K$ is even, so cannot be one.

 Case 2: $p=q=3 \bmod 4$: 
Here there is no unit of negative norm.  By Lemma~\ref{fund unit when D=1 mod 4},   
the fundamental unit  is of the form $\lambda=m+\omega = \frac{2m+1+\sqrt{D}}{2}$, and since  it has norm $N(\lambda)=+1$, and we have
\[
+1 = N(\lambda) = \frac{(2m+1)^2-D}{4} \longleftrightarrow D=(2m+1)^2-4.
\]
This is one of the family of ``narrow Richaud-Degert" extensions,  and we showed in Proposition~\ref{prop:Degert} that these never satisfy condition H if $D>5$. 

This concludes the proof of Theorem~\ref{main thm}.


\end{document}